\documentclass[a4,11pt]{article}
\usepackage{amsmath,amsfonts,amsthm,latexsym,amssymb,mathrsfs}

\newtheorem{df}{Definition}[section]
\newtheorem{thm}[df]{Theorem}
\newtheorem{pro}[df]{Proposition}

\newtheorem{rema}[df] {Remark}
\newtheorem{lem}[df] {Lemma}
\def\sfstp{{\hskip-1em}{\bf.}{\hskip1em}}

\title{ \bf  Factorizations of EP Banach space operators \\
and EP Banach algebra elements \/}

\author { Enrico Boasso}

\date{   }

\begin{document}

\maketitle \thispagestyle{empty} 


\setlength{\baselineskip}{14pt}

\begin{abstract}\noindent EP Banach space operators and
EP Banach algebra elements are characterized using different
kinds of factorizations. The results obtained generalize
well-known character-
izations of EP matrices, EP Hilbert space
operators and EP $C^*$-algebra elements. Furthermore, new results that
hold in these contexts are presented.\par

\end{abstract}

\section{\sfstp Introduction}\setcounter{df}{0}
\
\indent A complex matrix $T$ is said to be EP, if it commutes with its
Moore-Penrose inverse $T^{\dag}$. 
Moreover, the notion under consideration was extended to Hilbert
space operators and $C^*$-algebra elements, and it consists in a 
generalization of normal matrices and operators,
see the introductory section of \cite{DKP}.
Furthermore, thanks to the concept of hermitian Banach algebra element, in \cite{R}
V. Rako\v cevi\'c extended the notion of Moore-Penrose inverse to elements of a Banach algebra,
which led to study EP Banach space operators and EP Banach algebra
elements, see \cite{B}.\par

\indent The relationships among EP matrices and operators and the product operation have been object of
a particular attention. On the one hand, several articles studied when the
product of two EP matrices, Hilbert or Banach space operators, or elements of a $C^*$-algebra or a
Banach  algebra is again EP, see \cite{B,K} and the bibliography of these articles. 
On the other hand, D. Drivarialis, S. Karamasios and D. Pappas in \cite{DKP}
and D. S. Djordjevi\' c, J. J. Koliha and I. Straskaba in \cite{DKS} have recently
characterized EP Hilbert space operators and EP $C^*$-algebra
elements respectively through several different factorizations.
Note  that one of the main lines of research concerning 
 EP matrices and EP operators consists in characterizing them  
through factorizations. \par

\indent The objective of the present article is to characterize EP Banach space
operators and EP Banach algebra elements using factorizations to extend
results of \cite{DKP, DKS} to the mentioned contexts. Actually, three different kind of
factorizatiosn will be considered. It is worth noticing that due to the lack of involution
on a Banach algebra, and in particular on the Banach algebra of bounded
and linear maps defined on a Banach space, the proofs not only are different
from the ones known for matrices or Hilbert space operators, but also they give a new
insight into the cases where the involution does exist. Furthermore, thanks to the approach developed in this work, new results
in the frames both of Banach algebras and of $C^*$-algebras will be presented.
\par

\section{\sfstp Preliminary definitions and results}

\
\indent From now on $X$ and $Y$ will denote Banach spaces and $L(X, Y)$ will stand for the
Banach algebra of all bounded and linear maps defined on $X$ with values 
in $Y$. As usual, when $X=Y$, $L(X,Y)$ will be
denoted by $L(X)$.
In addition, if $T\in L(X,Y)$, then $N(T)$ and $R(T)$ will stand for the null space and the
range of $T$ respectively. \par
 
\indent On the other hand, $A$ will denote a unital Banach algebra and $e\in A$
will stand for the unit element of $A$. 
If $a\in A$, then $L_a, R_a \colon A\to A$
will denote the maps defined by
left  and right multiplication respectively, that is, $L_a(x)=ax$ and $R_a(x)=xa$,   
where $x\in A$. Moreover, the following notation will be used: 
$$
N(L_a)= a^{-1}(0),\hskip.3truecm R(L_a)=aA, \hskip.3truecm N(R_a)= a_{-1}(0), \hskip.3truecm R(R_a)=Aa.
$$

\indent Recall that an element $a\in A$ is said to be \it{regular}, \rm 
if it has a \it{generalized inverse}, \rm namely if there exists $b\in A$ such that
$a=aba$. Furthermore, a generalized inverse $b$ of a regular
element $a\in A$ will be said to be \it{normalized}, \rm if $b$ is regular
and $a$ is a generalized inverse of $b$, equivalently,
$a=aba$ and $b=bab$. 
Note that if $b$ is a generalized inverse of $a$,
then $c=bab$ is a normalized generalized inverse
of $a$. \par

\indent Next follows the key notion in the definition of 
Moore-Penrose invertible Banach algebra elements,
see \cite{V}.\par

\begin{df} \label{df1}Given a unital Banach algebra $A$, an element $a\in A$ is said to be hermitian,
if $\parallel exp(ita)\parallel =1$,  for all $ t\in\Bbb R$.
\end{df}

\indent Recall that  if $A$ is a $C^*$-algebra, then
$a\in A$ is hermitian if and only if $a$ is self-adjoint,
see \cite[Proposition 20, Chapter I, Section 12]{BD}. 
Moreover, $\mathcal{H}=\{a\in A\colon \hbox{  }$a$ \hbox{ is 
hermitian}\}\subseteq A$ is a closed linear vector space over the real field,
see \cite{V,D}. Since $A$ is unital, $e\in \mathcal{H}$,
which implies that $a\in \mathcal{H}$ if and only if $e-a\in \mathcal{H}$.
As regard equivalent definitions and the main properties of hermitian Banach 
algebra elements and hermitian Banach space operators, see \cite{BD, D, L, P, V}.
\par

\indent In \cite{R} V. Rako\v cevi\'c introduced the notion of Moore-Penrose invertible Banach algebra
element. Next the definition of this object  will be considered.\par

\begin{df}\label{df2} Let $A$ be a unital Banach algebra and consider $a\in A$. If there exists a normalized
generalized inverse $x\in A$ of $a$ such that $ax$ and $xa$ are hermitian elements of $A$,
then  $x$ will be said to be the Moore-Penrose inverse of $a$ and it will be
denoted by $a^{\dag}$.
\end{df}

\indent In the conditions of  Definition \ref{df2}, recall that according to \cite[Lemma 2.1]{R}, there exists at most one Moore-Penrose
inverse of $a\in A$.  In addition, if $a\in A$ has a Moore-Penrose inverse,
then $(a^{\dag})^{\dag}$ exists. In fact, $(a^{\dag})^{\dag}=a$.
Concerning the properties of the Moore-Penrose inverse 
in the frames
of Banach space operators and Banach algebras, 
see \cite{R, R2, B}, in $C^*$-algebras see \cite{HM, HM2, M}, for the original definition see \cite{Pe}.  \par
 
\indent In order to study the factorization that will be considered in the next section, the notion
of Moore-Penrose inverse for operators defined between different Banach spaces need to be introduced.
First of all, however, some preliminary results will be recalled.

\begin{rema} \label{rem101}\rm Let  $X$ and $Y$ be two Banach spaces and consider $T\in L(X,Y)$ and $S\in L(Y,X)$ 
such that $S$ is a normalized generalized inverse of $T$, i.e.,
$T=TST$ and $S=STS$. Then, it is not difficult to verify  the following facts:\par 
\noindent (i) $ST\in L(X)$ is an idempotent, $N(ST)=N(T)$, $R(ST)=R(S)$ and
$X=N(T)\oplus R(S)$.\par
\noindent (ii) $TS\in L(Y)$ is an idempotent,  $N(TS)=N(S)$, $R(TS)=R(T)$ and $Y=N(S)\oplus R(T)$.\par
\end{rema}

\begin{df}\label{def2.3} Let $X$ and $Y$ be two Banach spaces and consider $T\in L(X,Y)$. The operator $T$ will be
said to be Moore-Penrose invertible, if there exists $S\in L(Y,X)$ such that
$T=TST$, $S=STS$, and $ST\in L(X)$ and $TS\in L(Y)$
are hermitian operators.\end{df}

\indent Before going on some basic results concerning the objects of Definition \ref{def2.3} will be
considered.  \par

\begin{lem}\label{lem2.4} Let $X$ and $Y$ be two Banach spaces and consider $T\in L(X,Y)$ and $S_i\in L(Y,X)$ such that
$T$ and $S_i$  complies the four conditions of Definition  \ref{def2.3}, $i=1,2$. 
Then $S_1=S_2$.
\end{lem} 
\begin{proof}  Adapt the proof of \cite[Lemma 2.1]{R} to the conditions of the Lemma. 
\end{proof} 

\indent Since according to Lemma \ref{lem2.4} $T\in L(X,Y)$ has at most one Moore-Penrose inverse,
when the Moore-Penrose inverse of $T$ exists, it will be denoted by $T^{\dag}$.  On the other hand,
note that in the next proposition, given $T\in L(X,Y)$, $X$ and $Y$ Banach spaces,
$T^*\in L(Y^*,X^*)$ will denote the adjoint map of $T$ and $X^*$ and $Y^*$ will stand for the
dual space of $X$ and $Y$ respectively. \par

\begin{pro} \label{prop2.5} Let $X$ and $Y$ be two Banach spaces and consider $T\in L(X,Y)$.\par
\noindent \rm{ (i)} \it Necessary and suffcicient for $T^{\dag}$ to exist is the fact that there are
two hermitian idempotents $P\in L(X)$ and $Q\in L(Y)$ such that $N(P)=N(T)$ 
and $R(Q)=R(T)$.\par
\noindent \rm{ (ii)} \it If $T^{\dag}\in L(Y,X)$ exists, then $(T^*)^{\dag}\in L(X^*,Y^*)$ also exists, moreover, 
   $(T^*)^{\dag}= (T^{\dag})^*$.\par
\noindent \rm{ (iii)} \it Suppose that $T^{\dag}\in L(Y,X)$ exists, and let $X'\subseteq X$ and $Y'\subseteq Y$ be two closed vector subspaces 
such that $T(X')\subseteq Y'$ and $T^{\dag}(Y')\subseteq X'$. If $T'=T\mid_{X'}^{Y'}\in L(X',Y')$ and 
$T^{\dag'}=T^{\dag}\mid_{Y'}^{X'}\in L(Y',X')$, then  the Moore-Penrose inverse of $T'$ exists and $(T')^{\dag }=T^{\dag '}$.\par 
\noindent \rm{ (iv)} \it In the conditions of statement \rm{ (iii)}\it, if $\tilde{T}\colon X/X'\to Y/Y'$ 
and $\tilde{T^{\dag}}\colon Y/Y'\to X/X'$ denote the operators induced by $T$ and $T^{\dag}$ respectively,
then  the Moore-Penrose inverse of $\tilde{T}$ exists and $(\tilde{T})^{\dag}=\tilde{T^{\dag}}$.\par
\end{pro}
\begin{proof} Adapt  the proofs of \cite[Theorem 6]{B}, \cite[Proposition 7]{B} and \cite[Theorem 10]{B}
to the present situation.
\end{proof}

\indent The final point of this section concerns  the definition of EP Banach algebra elements.\par

\begin{df}\label{del2.6} Let $A$ be a unital Banach algebra. Given $a\in A$, $a$ will be said to be EP, if $a^{\dag}$ exists
and  $aa^{\dag}=a^{\dag}a$.
\end{df}
\indent Properties, characterizations and other facts regarding EP Banach space operators and EP Banach algebra elements
were studied in \cite{B}. In the following remark some of the most relevant results on these objects will be recalled.\par

 \begin{rema}\label{rema102}\rm Let $A$ be a unital Banach algebra and consider $a\in A$. \par
\noindent (i) Note that $a\in A$
is EP if and only if $a^{\dag}$ is EP. \par
\noindent (ii) According to \cite[Remark 12]{B}, necessary and
sufficient for $a\in A$ to be EP is the fact that $L_a\in L(A)$ is EP. \par
\noindent (iii) Let $A=L(X)$, $X$ a Banach space, and considet $T\in L(X)$. Then, according to \cite[Theorem 16]{B},
$T$ is EP if and only if $R(T)=R(T^{\dag})$ or $N(T)=N(T^{\dag})$.
\end{rema}

\section{\sfstp Factorization of the form $a=bc$}
\
\indent In this section, given a unital Banach algebra $A$, EP elements of 
the form $a=bc$ will be characterized, where $a,b,c\in A$, $a$ is Moore-Penrose invertible, $b^{-1}(0)=0$ and $cA=A$.
Concerning this kind of factorization, see the introductory section of \cite{DKP}.
In addition, compare the results of this section with \cite[section 5]{DKP} and \cite[sections 1.3, 2.3]{DKS}. However,
to prove the main results of this section, some preliminary facts
must be considered. Note that in what follows the identity map on the Banach space $X$ (respectively $Y$)
will be denoted by $I\in L(X)$ (respectively $I'\in L(Y)$).\par

\begin{pro}\label{pro1}Let $X$ and $Y$ be two Banach spaces and consider $T\in L(X)$, $C\in L(X,Y)$ and $B\in L(Y,X)$ such that $C$ is surjective, $B$ is injective
and $T=BC$. Suppose, in addition, that $T^{\dag}$ exists. Then, the following statements hold.\par
\noindent \rm (i) \it There exists $B^{\dag}\in L(X,Y)$ such that  $B^{\dag}B=I'$,\par
\noindent  \rm (ii) \it there exists $C^{\dag}\in L(Y,X)$ such that $CC^{\dag}=I'$,\par
\noindent \rm (iii) $T^{\dag}=C^{\dag}B^{\dag}$, $TT^{\dag}=BB^{\dag}$, $T^{\dag}T=C^{\dag}C$, $B^{\dag}=CT^{\dag}$ and $C^{\dag}= T^{\dag}B$.\par
\end{pro}

\begin{proof}  According to \cite[Theorem 6(ii)]{B},
there exist two hermitian idempotents $P$ and $Q\in L(X)$ such that $N(P)=N(T)$ and $R(Q)=R(T)$.
Since $N(B)=0$ and $R(B)=R(T)$, according to Proposition \ref{prop2.5}(i),
$B^{\dag}\in L(X,Y)$ exists. Actually, $I'\in L(Y)$ and $Q\in L(X)$ are two hermitian
idempotents such that $N(B)= N(I')$ and $R(B)=R(Q)$. 
Furthermore, according to Remark \ref{rem101},
$R(I'-B^{\dag}B)=N(B^{\dag}B)=N(B)=0$, i. e., $B^{\dag}B=I'$.\par

\indent Similarly, since $R(C)=Y$ and $N(C)=N(T)$, according to Proposition \ref{prop2.5}(i),
$C^{\dag}\in L(Y,X)$ exists. In fact, $I'\in L(Y)$ and $P\in L(X)$ are two hermitian
idempotents such that $R(C)=R(I')$ and $N(C)=N(P)$. Moreover, note that 
according to Remark \ref{rem101}, $R(CC^{\dag})=R(C)$.
However, since $CC^{\dag}$
and $I'$ are hermitian idempotents of $L(Y)$ whose ranges coincide, according to 
\cite[Theorem 2.2]{P}, $CC^{\dag}=I'$.\par

\indent Next consider $S=C^{\dag}B^{\dag}\in L(X)$. A straightforward calculation proves that   
$T=TST$, $S=STS$, $TS=BB^{\dag}$ and $ST=C^{\dag}C$. However, since
$BB^{\dag}$ and $C^{\dag}C$  are two hermitian idempotents, according to \cite[Lemma 2.1]{R},
$S=T^{\dag}$. The remaining two identities can be derived from statements (i) and (ii).
 \end{proof}
In the following theorems
Moore-Penrose invertible operators of the form $T=BC$ will be characterized.\par
 
\begin{thm}\label{thm3.2} In the conditions of Proposition \ref{pro1}, the following statements are equivalent.\par
\noindent \rm (i) \it $T$ is an EP operator,\par
\noindent \rm (ii) \it $BB^{\dag}=C^{\dag}C$,\par
\noindent \rm (iii) \it $N(B^{\dag})=N(C)$,\par
\noindent \rm (iv) \it $R(B)=R(C^{\dag})$.\par
\end{thm}
\begin{proof} According to Proposition \ref{pro1}(iii),  
statements (i) and (ii) are equivalent. In addition, since $N(T)=N(C)$, $N(T^{\dag})= N(B^{\dag})$, $R(T)=R(B)$ and $R(T^{\dag})=R(C^{\dag})$, according to \cite[Theorem 16]{B},
$T$ is EP if and only if statements (iii)-(iv)   holds.
\end{proof}

\begin{rema}\label{rem3.3}\rm Note that when $H$ is a Hilbert space and $T\in L(H)$, in the conditions of \cite[Theorem 5.1]{DKP}, necessary and 
sufficient for $T$ to be EP is the fact that $R(B)= R(C^*)$. In fact, apply Theorem \ref{thm3.2} and use that $R(C^*)=R(C^{\dag})$,
see section 2 of \cite{DKP}.
\end{rema}

\begin{thm} \label{thm3.4}  In the conditions of Proposition \ref{pro1}, necessary and sufficient for 
$T$ to be EP is that one the following statements holds.\par
\noindent  \rm (i) \it $(I-C^{\dag}C)B=0$ and $C(I-BB^{\dag})=0$, \par
\noindent \rm (ii) \it $B^{\dag}(I-C^{\dag}C)=0$ and $C(I-BB^{\dag})=0$, \par
\noindent \rm (iii) \it $(I-C^{\dag}C)B=0$ and $(I-BB^{\dag})C^{\dag}=0$,\par
\noindent \rm (iv) \it $B^{\dag}(I-C^{\dag}C)=0$ and $(I-BB^{\dag})C^{\dag}=0$,\par
\noindent  \rm (v) \it $C(I-BB^{\dag})=0$ and $B^{\dag}(I-C^{\dag}C)B=0$,\par
\noindent \rm (vi) \it $B^{\dag}(I-C^{\dag}C)=0$ and $C(I-BB^{\dag})C^{\dag}=0$.
\end{thm}
\begin{proof} Suppose that $T$ is EP. Then, according to Theorem \ref{thm3.2}, $R(B)=R(C^{\dag})$ and
$N(B^{\dag})=N(C)$. However, since $N(I-C^{\dag}C)=R(C^{\dag}C)=R(C^{\dag})$
and $R(I-BB^{\dag})=N(BB^{\dag})=N(B^{\dag})$ (Remark \ref{rem101}), statement (i) holds.\par

\indent On the other hand, note that the conditions in statement (i) are equivalent to
$R(T)=R(B)\subseteq N(I-C^{\dag}C)=R(C^{\dag})=R(T^{\dag})$ and
$N(T^{\dag})=N(B^{\dag})=R(I-BB^{\dag})\subseteq N(C)=N(T)$.
However, since according to Remark \ref{rem101} $N(T^{\dag})\oplus R(T)=X=N(T)\oplus R(T^{\dag})$,
it is not difficult to prove that $N(T)=N(T^{\dag})$ and $R(T)=R(T^{\dag})$. In particular, according to
\cite[Theorem 16]{B},  $T$ is  EP.\par 
\indent The equivalence among the condition of being EP and statements (ii)-(iv) can be proved using similar arguments.\par

\indent Next consider statement (v).  According to what has been proved, if $T$ is EP, then  $C(I-BB^{\dag})=0$.
In addition, if $T$ is EP, then $T^{\dag}T=T^{\dag}T^{\dag}TT$, equivalently, 
$C^{\dag}C=C^{\dag}B^{\dag}C^{\dag}CBC$. Now well, according to Proposition \ref{pro1}(i)-(ii),
$B^{\dag}B=I'=B^{\dag}C^{\dag}CB$, which is equivalent to $B^{\dag}(I-C^{\dag}C)B=0$.\par

\indent On the other hand, if statement (v) holds, then according to what has been proved,
$N(B^{\dag})\subseteq N(C)$. However, since $X= N(B^{\dag})\oplus R(B)$ (Remark \ref{rem101}), 
if in addition $B^{\dag}(I-C^{\dag}C)B=0$, then a straightforward calculation  proves that $B^{\dag}(I-C^{\dag}C)=0$. 
Since $R(I-C^{\dag}C) =N(C)$ (Remark \ref{rem101}), $N(C)\subseteq N(B^{\dag})$.
Therefore, according to Theorem \ref{thm3.2}(ii), $T$ is EP.\par
\indent The equivalennce between the condition of being EP and statement (vi) can be proved 
in a similar way.
\end{proof}

\begin{thm}\label{thm3.5}  In the conditions of Proposition \ref{pro1},  the following statements are equivalent.\par
\noindent \rm (i) \it $T$ is an EP operator,\par
\noindent \rm (ii) \it there exists an isomorphism $U\in L(Y)$ such that $C=UB^{\dag}$,\par
\noindent \rm (iii) \it there exists an injective map $U_1\in L(Y)$ such that $C=U_1B^{\dag}$,\par
\noindent \rm (iv) \it there exist $U_2, U_3\in L(Y)$ such that $C=U_2B^{\dag}$ and $B^{\dag}=U_3C$,\par 
\noindent \rm (v) \it there exists an isomorphism $W\in L(Y)$ such that $B=C^{\dag}W$,\par
\noindent \rm (vi) \it there exists a surjective map $W_1\in L(Y)$ such that $B=C^{\dag}W_1$,\par
\noindent \rm (vii) \it there exist $W_2, W_3\in L(Y)$ such that $B=C^{\dag}W_2$ and $C^{\dag}=BW_3$,\par
\noindent \rm (viii) \it there exist $H_1, H_2\in L(Y)$ such that $B=C^{\dag}H_2$ and $C=H_1B^{\dag}$,\par
\noindent \rm (ix) \it there exist $ K_1, K_2\in L(Y)$ such that $C^{\dag}=BK_2$ and $B^{\dag}=K_1C$,\par
\noindent \rm (x) \it there exists an injective map $S_1\in L(Y)$ such that $B^{\dag}=S_1C$,\par
\noindent \rm (xi) \it there exists a surjective map $S_2\in L(Y)$ such that $C^{\dag}=BS_2$.\par
\end{thm} 

\begin{proof} Suppose that $T$ is EP and define $U=CTC^{\dag}\in L(Y)$. According to Proposition \ref{pro1}(ii)
and the second identity of Theorem \ref{thm3.4}(i), $U=CB$ and $UB^{\dag }=C$. To prove that $U\in L(Y)$ is an isomorphic map,
consider $Z= B^{\dag }C^{\dag }\in L(Y)$. According to the second identity of Theorem \ref{thm3.4}(i)
and Proposition \ref{pro1}(ii), $UZ=I'$. In addition, according to the first identity of Theorem \ref{thm3.4}(i) and
Proposition \ref{pro1}(i), $ZU=I'$.\par

\indent It is clear that statement (ii) implies statement (iii). On the other hand, if statement (iii) holds,
then $N(C)=N(B^{\dag})$, which, according to Theorem \ref{thm3.2}(iii), is equivalent to statement (i).\par

\indent Clearly,  statement (ii) implies statement (iv), which in turn implies that $N(C)=N(B^{\dag})$.
Consequently, according to Theorem \ref{thm3.2}(iii), $T$ is an EP operator.\par

\indent On the other hand, if $T$ is EP, then, according to the first identity of Theorem \ref{thm3.4}(i),
$C^{\dag }U=C^{\dag }CB=B$. As a result, statement (v) holds with $W=U$.\par

\indent Statement (v) implies statement (vi), which in turn implies that $R(B)=R(C^{\dag})$.
However, according to Theorem \ref{thm3.2}(iv), $T$ is EP.\par

\indent In addition,  statement (v) implies statement (vii), which in turn implies  that $R(B)=R(C^{\dag})$.
In particular, according to Theorem \ref{thm3.2}(iv), $T$ is EP.\par

\indent It is clear that  statements (iv) and (vii) implies statements (viii) and (ix). To prove that
both statement (viii) and statement (ix) implies that $T$ is EP, consider the decomposition of $X$ defined by
$TT^{\dag}=BB^{\dag}$ and $T^{\dag}T=C^{\dag}C$, i.e., $X=N(B^{\dag})\oplus R(B)= N(C)\oplus R(C^{\dag})$.
 Now well, if statement (viii) (respectively statement (ix)) holds,
then  $N(B^{\dag})\subseteq N(C)$ and $R(B)\subseteq R(C^{\dag})$
(respectively $N(C)\subseteq  N(B^{\dag})$ and $R(C^{\dag})\subseteq R(B)$).
However, according the decompositions of $X$, 
if statement (viii) or statement (ix) holds, then
it is not difficult to prove that $N(B^{\dag})= N(C)$ and $R(B)= R(C^{\dag})$.
Therefore, according to Theorem \ref{thm3.2},
$T$ is EP.\par
\indent Next, statement (ii) implies statement (x). On the other hand, if statement (x) holds, since $B^{\dag}=S_1C$,
$Y=R(B^{\dag})\subseteq R(S_1)$. However, since $N(S_1)=0$, statement (ii) holds.\par
\indent Similarly, statement (v) implies statement (xi). To prove the converse,
if $C^{\dag}=BS_2$, then, according to Proposition \ref{pro1}, $N(S_2)=0$. Since $S_2$ is surjective, statement (v) holds.
\end{proof}

\indent Next the Banach algebra case will be studied. Firstly some preliminary
facts will be considered.\par

\begin{pro}\label{prop3.6} Let $A$ be a unital Banach algebra and consider $a$, $b$, $c\in A$
such that $a^{\dag}$ exists, $a=bc$, $b^{-1}(0)=0$ and $cA=A$. Then, the following statements hold.\par  
\noindent \rm (i) \it There exists $b^{\dag}\in A$ such that $b^{\dag}b=e$,\par
\noindent \rm (ii) \it there exists $c^{\dag}\in A$ and $cc^{\dag}=e$,\par
\noindent \rm (iii) $a^{\dag}=c^{\dag}b^{\dag}$, $aa^{\dag}=bb^{\dag}$,
$a^{\dag}a=c^{\dag}c$, $b^{\dag}=ca^{\dag}$ and $c^{\dag}= a^{\dag}b$.
\end{pro}  
\begin{proof} Consider the maps $L_a$, $L_b$, $L_c\in L(A)$. It is clear that
$L_a=L_bL_c$, $N(L_b)=0$ and $R(L_c)=A$. What is more, according to \cite[Remark 5(ii)]{B},
$(L_a)^{\dag}$ exists and $(L_a)^{\dag}=L_{a^\dag}\in L(A)$. 
Therefore, according to Proposition \ref{pro1}, $(L_b)^{\dag}$ exists  and $(L_b)^{\dag}L_b=I_A$,
$(L_c)^{\dag}$ exists and $L_c(L_c)^{\dag}=I_A$, 
$L_{a^{\dag}}=(L_c)^{\dag}(L_b)^{\dag}$, $(L_b)^{\dag}=L_{ca^{\dag}}$ and $(L_c)^{\dag}= L_{a^{\dag}b}$,
where $I_A\in L(A)$ denotes the identity map of $A$.\par
\indent Next consider $c'=(L_c)^{\dag}(e)=a^{\dag}b$. Since $L_c(L_c)^{\dag}=I_A$, 
$cc'=e$. As a result, $c=cc'c$ and $c'=c'cc'$. In addition,
$c'c=a^{\dag}bc=a^{\dag}a$, which is a hermitian idempotent. 
Consequently, according to \cite[Lemma 2.1]{R}, $c^{\dag}$
exists, $c^{\dag}=c'=a^{\dag}b$ and $cc^{\dag}=e$.\par

\indent Let $b'=ca^{\dag}$. Since  $c^{\dag}=a^{\dag}b$ and  $cc^{\dag}=e$,
$b'b=e$. In particular, $b=bb'b$ and $b'=b'bb'$. Moreover,
$bb'=bca^{\dag}=aa^{\dag}$, which is a hermitian idempotent.
Thus, according to \cite[Lemma 2.1]{R}, $b^{\dag}$ exists,
$b^{\dag}=b'=ca^{\dag}$ and $b^{\dag}b=e$.\par

\indent Finally, define $a'=c^{\dag}b^{\dag}$. Then, since $b^{\dag}b=e$ and $cc^{\dag}=e$,
$aa'a=a$, $a'aa'=a'$, $aa'=bb^{\dag}$ and $a'a=c^{\dag}c$. Therefore, according to
 \cite[Lemma 2.1]{R}, $a^{\dag}=a'=c^{\dag}b^{\dag}$.
\end{proof}
\indent In the following theorem, given $A$ a unital Banach algebra, $A^{-1}$
will stand for the set of all invertible elements of $A$.\par
\begin{thm}\label{thm3.7}In the conditions of Proposition \ref{prop3.6},
the following statements are equivalent.
\begin{align*}
&\hbox{\rm (i) } a\in A \hbox{  is EP},&&\hbox{\rm (ii) } bb^{\dag}=c^{\dag}c,&\\
&\hbox{\rm (iii) } (b^{\dag})^{-1}(0)=c^{-1}(0),& &\hbox{\rm (iv) } bA=c^{\dag}A,&\\
&\hbox{\rm (v) } b_{ -1}(0)=(c^{ \dag})_{ -1}(0), &&\hbox{\rm (vi) } Ac=Ab^{ \dag},\\
&\hbox{\rm (vii) } (e-c^{\dag}c)b=0 \hbox{ and } c(e-bb^{\dag})=0,&
&\hbox{\rm (viii) } b^{\dag}(e-c^{\dag}c)=0 \hbox{ and } c(e-bb^{\dag})=0,&\\
&\hbox{\rm (ix) } (e-c^{\dag}c)b=0 \hbox{ and } (e-bb^{\dag})c^{\dag}=0,&
&\hbox{\rm (x) }  b^{\dag}(e-c^{\dag}c)=0 \hbox{ and } (e-bb^{\dag})c^{\dag}=0,&\\
&\hbox{\rm (xi) } c(e-bb^{\dag})=0 \hbox{ and } b^{\dag}(e-c^{\dag}c)b=0,&
&\hbox{\rm (xii) } b^{\dag}(e-c^{\dag}c)=0 \hbox{ and } c(e-bb^{\dag})c^{\dag}=0,&\\
&\hbox{\rm (xiii) } \exists\hbox{ } x\in A^{-1}: c=xb^{\dag},&
&\hbox{\rm (xiv) }\exists \hbox{ }y\in A: y^{-1}(0)=0  \hbox{ and } c=yb^{\dag}, &\\
&\hbox{\rm (xv) } \exists  \hbox{ } z_1, z_2\in A: c=z_1b^{\dag}  \hbox{ and }
 b^{\dag}=z_2c,&&\hbox{\rm (xvi) }\exists \hbox{ } u\in A^{-1}: b=c^{\dag}u,&\\
&\hbox{\rm (xvii) } \exists \hbox{ }v\in A: vA=A  \hbox{ and } b=c^{\dag}v, &
&\hbox{\rm (xviii) }\exists  \hbox{ } w_1, w_2\in A: b=c^{\dag}w_1  \hbox{ and }
 c^{\dag}=bw_2,\\
&\hbox{\rm (xix) }\exists \hbox{ } h_1, h_2\in A: b=c^{\dag}h_2 \hbox{ and }  c=h_1b^{\dag},&
&\hbox{\rm (xx) }\exists \hbox{ } k_1, k_2\in A: c^{\dag}=bk_2 \hbox{ and } b^{\dag}=k_1c,&\\
&\hbox{\rm (xxi) }\exists s_1\in A: s_1^{-1}(0)=0 \hbox{ and }  b^{\dag}=s_1c,&
&\hbox{\rm (xxii) }\exists s_2\in A: s_2A=A \hbox{ and } c^{\dag}=bs_2,&\\
&\hbox{\rm (xxiii) } bA^{ -1}=c^{ \dag}A^{ -1},& &\hbox{\rm (xxiv) }A^{ -1}c=A^{ -1}b^{ \dag},&\\
&\hbox{\rm (xxiv) }a\in c^{ \dag}A\cap Ab^{ \dag},& &\hbox{\rm (xxvi) } a^{ \dag}\in bA\cap Ac.&\\
\end{align*}
\end{thm}
\begin{proof}
Let $T=L_a$, $B=L_b$ and $C=L_c\in L(A)$. 
Then, according to Proposition \ref{prop3.6} and
\cite[Theorem  5(ii)]{B},
$T$, $B$ and $C$ are Moore-Penrose invertible operators, what is more,
$T^{\dag}=L_{a^{\dag}}$, $B^{\dag}=L_{b^{\dag}}$ and $C^{\dag}=L_{c^{\dag}}$.
Recall also that according to
\cite[Remark 12]{B}, $a\in A$ is EP if and only if $L_a\in L(A)$ is EP.
Therefore, since  $T=BC$, $N(B)=0$, $R(C)=A$, according to 
Theorem \ref{thm3.2} and Theorem \ref{thm3.4}, it is easy to prove that statements (i)-(iv) and  (vii)-(xii) are equivalent.\par

\indent On the other hand, according to Proposition \ref{prop3.6}(i)-(ii),  
$a_{ -1}(0)=b_{ -1}(0)$ and $(a^{ \dag})_{ -1}(0)=(c^{ \dag})_{ -1}(0)$. Consequently, according to \cite[Theorem 18(viii)]{B},
statements (i) and (v) are equiva-
lent. In addition, a straightforward
calculation proves that $R_a=R_{a^{\dag}a}=R_{c^{\dag}c}=R_c$.
Similarly, $R_{a^{\dag}}=R_{aa^{\dag}}=R_{bb^{\dag}}=R_{b^{\dag}}$.
Consequently, $Aa=Ac$ and $Aa^{ \dag}=Ab^{ \dag}$. 
Therefore,  according to \cite[Theorem 18(ix)]{B},
statements (i) and (vi) are equivalent. \par

\indent To prove that statement (i) and statements (xiii) -(xxii) are equivalent, use as before
the multiplication operators and adapt the proof of Theorem \ref{thm3.5} to the present
situation. Note that $x=u=cb$.\par
\indent Observe that statement (xxiii) (respectively  (xxiv))
is equivalent to statement (xvi) (respectively (xiii)).\par

\indent Finally, according to  \cite[Theorem 18(xiii)-(xiv)]{B},
 statement (i) and statements (xxv)-(xxvi) are equivalent. 
\end{proof}

\indent In the frame of unital $C^*$-algebras, the  results of Theorem \ref{thm3.7} can be
reformulated using the adjoint instead of the
Moore-Penrose inverse. However, to this end some preparation
is needed. \par

\begin{rema}\label{rem105}\rm Recall  that given a unital $C^*$-algebra $A$ and
$x\in A$, then $x^{ \dag} A=x^*A$, $(x^{ \dag})^{ -1}(0)= (x^*)^{ -1}(0)$,
$Ax^{ \dag} =Ax^*$ and $(x^{ \dag})_{ -1}(0)= (x^*)_{ -1}(0)$,
see  \cite[Lemma 1.5]{K}. 
\end{rema}

\indent Compare the following lemma with  \cite[Lemma 1.5]{K}.\par

\begin{lem}\label{lem3.8} Let $A$ be a unital $C^*$-algebra and consider
$a\in A$ such that $a$ is Moore-Penrose invertible.\par
\noindent (i) If $p= aa^{ \dag}$, then $v=e-p+(a^{ \dag})^*a^{ \dag}\in A^{ -1}$,
 $a^{ \dag}=a^*v$ and $aa^*v=vaa^*=p$.\par
\noindent (ii) If $q=a^{ \dag}a$, then $w=e-q +a^{ \dag}(a^{ \dag})^*\in A^{ -1}$,
 $a^{ \dag}=wa^*$ and $wa^*a=a^*aw=q$.\par
\indent Suppose, in addition, that $a=bc$, where $b$ and $c$ are such that
$b^{ -1}(0)=0$ and $cA=A$. Then,\par
\noindent (iii) $b^{ \dag}(b^{ \dag})^*\in A^{ -1}$ and $(b^{ \dag}(b^{ \dag})^*)^{ -1}=b^*b$,\par
\noindent (iv) $(c^{ \dag})^* c^{ \dag}\in A^{ -1}$ and  $((c^{ \dag})^* c^{ \dag})^{ -1}=cc^*$,\par
\noindent (v) $vb=(b^{ \dag})^* (c^{ \dag})^* c^{ \dag}$ and $cw=b^{ \dag}(b^{ \dag})^*(c^{ \dag})^*$. 
\end{lem}
\begin{proof} Note that $(a^{ \dag})^*a^{ \dag}$ and $aa^*$ belong to
the subalgebra $pAp$. What is more, $(a^{ \dag})^*a^{ \dag}aa^*=aa^*(a^{ \dag})^*a^{ \dag}=p$.
Consequently, $v\in A^{ -1}$ and a straightforward calculation
proves that $a^{ \dag}=a^*v$ and $aa^*v=vaa^*=p$.\par

\indent Interchanging $a$ with $a^*$, statement (ii) can be derived from statement (i).\par
\indent  Statements (iii)-(iv) can be easily derived from  Proposition \ref{prop3.6}(i)-(ii).\par
\indent Note that $v=e-bb^{ \dag} +(b^{ \dag})^*(c^{ \dag})^*c^{ \dag}b^{ \dag}$.
Consequently, according to Proposition \ref{prop3.6}(i), 
$vb=(b^{ \dag})^*(c^{ \dag})^*c^{ \dag}$. A similar argument, using 
 Proposition \ref{prop3.6}(ii),  proves the remaining
identity.
\end{proof}
\begin{thm}\label{thm3.9}Let $A$ be a unital $C^*$-algebra and consider $a\in A$
such that $a^{ \dag}$ exists. Suppose that there exist $b,c\in A$ such that $a=bc$, $b^{-1}(0)=0$ and $cA=A$. Then,
the following statements are equivalent.
\begin{align*}
&\hbox{\rm (i) } a\in A \hbox{  is EP},&&\hbox{\rm (ii) } a\in c^*A\cap Ab^*,&\\
&\hbox{\rm (iii) } (b^*)^{-1}(0)=c^{-1}(0),& &\hbox{\rm (iv) } bA=c^*A,&\\
&\hbox{\rm (v) } b_{ -1}(0)=(c^*)_{ -1}(0), &&\hbox{\rm (vi) } Ac=Ab^*,\\
&\hbox{\rm (vii) } bA^{ -1}=c^*A^{ -1},& &\hbox{\rm (viii) }A^{ -1}c=A^{ -1}b^*,&\\
&\hbox{\rm (ix) } \exists\hbox{ } x\in A^{-1}: c=xb^*,&
&\hbox{\rm (x) }\exists \hbox{ }y\in A: y^{-1}(0)=0  \hbox{ and } c=yb^*, &\\
&\hbox{\rm (xi) } \exists  \hbox{ } z_1, z_2\in A: c=z_1b^*  \hbox{ and }
 b^*=z_2c,&&\hbox{\rm (xii) } \exists \hbox{ }v\in A: vA=A  \hbox{ and } b=c^*v, &\\
&\hbox{\rm (xiii) }\exists s_1\in A: s_1^{-1}(0)=0 \hbox{ and }  b^*=s_1c,&
&\hbox{\rm (xiv) }\exists s_2\in A: s_2A=A \hbox{ and } c^*=bs_2.&\\
\end{align*}
\end{thm}
\begin{proof} The equivalence among statements (i)-(vi) can be derived from Theorem \ref{thm3.7}
and the relationships among the null spaces and the ranges of the operators $L_{ x^{ \dag}}$
and $L_{ x^*}\in L(A)$, and of the operators $R_{ x^{ \dag}}$ and 
$R_{ x^*}\in L(A)$, for $x= b$ and $c$ (Remark \ref{rem105}).\par

\indent Note that statement (ix) is equivalent to the fact that there exists $z\in  A^{-1}$
such that $b=c^*z$, which in turn is equivalent to statements (vii) and (viii). 
Now well, if there exists $z\in A^{-1}$
such that $b=c^*z$, then $bA=c^*A=c^{\dag}A$ which, according to Theorem \ref{thm3.7}(iv),
implies that $a$ is EP. On the other hand, if $a$ is EP, according to Theoren \ref{thm3.7}(xvi),
there exists $u\in A^{-1}$ such that $b=c^{\dag}u$. Then, according to  
Proposition \ref{prop3.6}(i) and statements (i) and (v) of Lemma \ref{lem3.8}, 
$$
b=c^{\dag}b^{\dag}bu=a^{\dag}bu=a^*vbu=c^*b^*(b^{ \dag})^* (c^{ \dag})^* c^{ \dag}u=c^*z,
$$
where $z=(c^{ \dag})^* c^{ \dag}u$.   However, according to Lemma \ref{lem3.8} (iv), $z\in A^{-1}$.\par
\indent Clearly, statement (ix) implies statement (x), which in turn implies statement (iii).\par
\indent Similarly, statement (ix) implies statement (xi), which in turn implies statement (iii).\par
\indent Statement (ix) implies statement (xii), which in turn implies statement (iv).\par
\indent Finally, statement (ix) implies statements (xiii) and (xiv). On the other hand,
statement (xiii) implies statement (iii) and statement (xiv) implies statement (iv). 
\end{proof}

\indent Before the next theorem, recall that if $A$ is a unital  $C^*$-algebra and $a\in A$,
then $(aa^*)^{-1}(0)=(a^*)^{-1}(0)$ and $(a^*a)^{-1}(0)=a^{-1}(0)$. Moreover,
if $a\in A$ is Moore-Penrose invertible, then  $aa^*A=aA$ and $a^*aA=a^*A$, 
see \cite[Lemma 1.1(ii)]{DKS}.\par
\begin{thm}\label{thm3.10} In the conditions of Theorem \ref{thm3.9}, the following statements
are equivalent.\par
\noindent (i) $a\in A$ is EP,\par
\noindent (ii) $a^*a= c^*b^*bcbb^{ \dag}$ and $a^*a=c^*b^*c^{ \dag}cbc$,\par
\noindent (iii) $aa^*=bcc^*b^*c^*(c^*)^{ \dag}$ and $aa^*= bcbb^{ \dag}c^*b^*$,\par
\noindent (iv) $a^*a= c^*b^*bcbb^{ \dag}$ and  $aa^*=bcc^*b^*c^*(c^*)^{ \dag}$,\par
\noindent (v) $aa^*= c^{ \dag}b^{ \dag}bcbcc^*b^*$ and $a^*a= c^*b^*bcbb^{ \dag}$,\par
\noindent (vi) $a^*a= bc c^{ \dag}b^{ \dag}b^*c^*bc$ and $aa^*=bcc^*b^*c^*(c^*)^{ \dag}$,\par
\noindent (vii) $aa^*= c^{ \dag}b^{ \dag}bcbcc^*b^*$ and $a^*a= bc c^{ \dag}b^{ \dag}b^*c^*bc$.
\end{thm}
\begin{proof} If $a\in A$ is EP, then $a=aaa^{ \dag}$ and $a=a^{ \dag}aa$. Using these
identities, statement (i) implies all the others.\par

\indent  On the other hand, suppose that statement (ii) holds. If  $a^*a= c^*b^*bcbb^{ \dag}$, then $(b^{ \dag})^{-1}(0)\subseteq 
(a^*a)^{-1}(0)=a^{-1}(0)=c^{ -1}(0)$. In addition, if $c^*b^*bc=a^*a=c^*b^*c^{ \dag}cbc$, according to
Propos-
ition \ref{prop3.6}(ii),  $b^*b=b^*c^{ \dag}cb$, equivalently,
$b^*(e-c^{ \dag}c)b=0$. Now well, since   $(b^*)^{ -1}(0)=(b^{ \dag})^{-1}(0)$ (\cite[Lemma 1.5]{K}),
according to the decomposition of $A$ defined by $bb^{ \dag}$,
i.e., $A=bA\oplus (b^{ \dag})^{ -1}(0)$, $b^*(e-c^{ \dag}c)=0$. Consequently, since $c^{-1}(0)=(c^{\dag}c)^{-1}(0)=(e-c^{\dag}c)A\subseteq (b^*)^{ -1}(0)$,
according to Theorem \ref{thm3.9}(iii), $a$ is EP.\par 

\indent To prove that statement (iii) implies that $a$ is EP, note that if  $aa^*=bcc^*b^*c^*(c^*)^{ \dag}$,
then according to \cite[Lemma 1.5]{K} and Proposition \ref{prop3.6}(ii),
$$
c^{-1}(0)=((c^*)^*)^{-1}(0)=((c^*)^{ \dag})^{-1}(0)\subseteq (aa^*)^{-1}(0)=(a^*)^{-1}(0)=(a^{\dag})^{-1}(0)=(b^{\dag})^{-1}(0).
$$
\indent Moreover, if $bcc^*b^*=aa^*= bcbb^{ \dag}c^*b^*$, then according to Proposition \ref{prop3.6}(i), 
$cc^*=cbb^{ \dag}c^*$, equivalently $c(e-bb^{ \dag})c^*=0$. Now well,  since $c^*A=c^{ \dag}A$ ( \cite[Lemma 1.5]{K}),
according to the decomposition of $A$ defined by $c^{ \dag}c$,
i.e., $A=c^{ \dag}A\oplus c^{ -1}(0)$, $c(e-bb^{ \dag})=0$, equivalently $(b^{ \dag})^{-1}(0)=(bb^{ \dag})^{-1}(0)=(e-bb^{ \dag})A\subseteq c^{-1}(0)$.
Therefore, since $(b^*)^{ -1}(0)=(b^{ \dag})^{-1}(0)$, according to Theorem \ref{thm3.9}(iii),
$a$ is EP.\par

\indent Next consider statement (iv). Note that according to what has been proved,
 $(b^*)^{-1}(0)=c^{-1}(0)$, which, according again to  Theorem \ref{thm3.9}(iii),
implies that $a$ is EP.\par

\indent If statement (v) holds, then according to \cite[Lemma 1.1(ii)]{DKS},  \cite[Lemma 1.5]{K} and the first identity
of statement (v),  $bA=aA=aa^*A\subseteq c^{ \dag}A=c^*A$.
What is more, according to what has been proved, the second identity of statement (v)  implies that
$(b^{ \dag})^{-1}(0)=(b^*)^{-1}(0)\subseteq c^{-1}(0)$. Consequently, according to the decompositons of $A$
defined by $c^{ \dag}c$ and $bb^{ \dag}$ considered above,
$(b^*)^{-1}(0)= c^{-1}(0)$ and $bA=c^*A$. Therefore, according to Theorem \ref{thm3.9}(iii)-(iv), $a$ is EP.\par 

\indent Next suppose that statement (vi) holds. Then, according to the first identity of statement (vi),
$c^*A=c^{ \dag}A=a^{ \dag}A=a^*A=a^*aA\subseteq bA$. Moreover, according to what
has been proved, the second identity of statement (vi) implies that $c^{-1}(0)\subseteq (b^*)^{-1}(0)=(b^{\dag})^{-1}(0)$.
Therefore, according to the decompositons of $A$
defined by $c^{ \dag}c$ and $bb^{ \dag}$ considered above,
 $c^{-1}(0)=(b^*)^{-1}(0)$ and $c^*A= bA$, which according to  Theorem \ref{thm3.9}(iii)-(iv),
is equivalent to the fact that $a$ is EP.\par

\indent If statement (vii) holds, according to what has been proved, $c^*A= bA$, which, according to 
 Theorem \ref{thm3.9}(iv), implies that $a$ is EP.
\end{proof}

\section{\sfstp  Factorization of the form $a^{\dag}=sa$ }
\
\indent In this section, given a unital Banach algebra $A$, EP elements
of the form $a^{\dag}=sa$ will be characterized, $a,s\in A$. Recall that given a Banach algebra $A$ and $a\in A$, according to \cite[Theorem 18(xviii)]{B}, 
necessary and sufficient for $a$ to be EP is the fact that there is $z\in A^{ -1}$ such that $a^{ \dag}=za$.
In what follows this result will be refined. Compare this section with \cite[section 4]{DKP} and \cite[sections 1.1, 2.1]{DKS}
\begin{thm}\label{thm4.1} 
Let $A$ be a unital Banach algebra and consider $a\in A$ such that $a^{ \dag}$ exists.
Then, the following statements are equivalent.\par
\noindent \rm (i) \it $a$ is EP,\par
\noindent \rm (ii) \it there exists $s\in A$ such that $s^{ -1}(0)=0$ and $a^{\dag}=sa$,\par
\noindent \rm (iii) \it there exist $s_1$ and $s_2 \in A$ such that $a^{\dag}=s_1a$ and $a=s_2a^{\dag}$,\par
\noindent \rm (iv) \it there exists $u\in A$ such that $uA=A$ and $a^{\dag}=au$,\par
\noindent \rm (v) \it there exist $u_1$ and $u_2\in A$ such that $a^{\dag}=au_1$ and $a=a^{\dag}u_2$,\par
\noindent \rm (vi) \it there exists $t\in A$ such that $t_{ -1}(0)=0$ and $a^{\dag}=at$,\par
\noindent \rm (vii) \it there exists $x\in A$ such that $Ax=A$ and $a^{\dag}=xa$,\par
\noindent \rm (viii) \it there exists $v\in A^{ -1}$ such that  $a^{\dag}a=vaa^{\dag}$,\par
\noindent \rm (ix) \it there exists $v_1\in A$ such that $v_1^{ -1}(0)=0$ and $a^{\dag}a=v_1aa^{\dag}$,\par
\noindent  \rm (x) \it there exist $v_2$ and $v_3\in A$ such that $a^{\dag}a=v_2 aa^{\dag}$
and $aa^{\dag} =v_3a^{\dag}a$,\par
\noindent \rm (xi) \it there exists $w\in A^{ -1}$ such that $a^{\dag}a=aa^{\dag}w$,\par
\noindent \rm (xii) \it there exists $w_1\in A$ such that $w_1A=A$ and  $a^{\dag}a=aa^{\dag}w_1$,\par
\noindent  \rm (xiii) \it there exist $w_2$ and $w_3\in A$ such that $a^{\dag}a= aa^{\dag}w_2$
and $aa^{\dag} =a^{\dag}aw_3$,\par
\noindent \rm (xiv) \it there exist $z_1$ and $z_2\in A$ such that $a^{\dag}a=az_1a^{\dag}$ and $aa^{\dag}=a^{\dag}z_2a$.\par
\end{thm}

\begin{proof} According to  \cite[Theorem 18(xviii)]{B}, if $a\in A$ is EP, then
statement (ii)  and (iii) hold. On the other hand, if statement (ii) or (iii) holds, then
$a^{-1}(0)=(a^{\dag})^{-1}(0)$. Consequently, according to \cite[Theorem 18(iii)]{B}, $a$
is EP.\par
\indent Similarly, if $a$ is EP, according to  \cite[Theorem 18(xvii)]{B}, statements (iv) and (v)
holds. If, on the other hand, statements (iv) or (v) holds, then $aA=a^{\dag}A$. In particular,
according to  \cite[Theorem 18(iv)]{B}, $a$ is EP.\par

\indent To prove the equivalences among statement (i) and statements (vi)-(vii)
apply arguments similar to the ones in the proofs of the equivalences among statement (i), (ii)
and (iv), and use  \cite[Theorem 18(xvii)]{B} and \cite[Theorem 18(viii)-(ix)]{B}.\par

\indent If $a$ is EP, clearly statements (viii)-(x) hold. On the other hand, if
one of the statements (viii)-(x) holds, then it is not difficult to prove that 
$a^{-1}(0)=(a^{\dag})^{-1}(0)$ (recall that according to Remark \ref{rem101}, $(aa^{\dag})^{-1}(0)=(a^{\dag})^{-1}(0)$
and $(a^{\dag}a)^{-1}(0) =a^{-1}(0) $). However, according to \cite[Theorem 18(iii)]{B}
 $a$ is EP.\par

\indent In a similar way, using in particular  \cite[Theorem 18(iv)]{B}, the equivalence among the condition of being EP
and statements (xi)-(xiii) can be proved.\par

\indent It is clear that if $a$ is EP, then statement (xiv) holds. 
On the other hand, statement (xiv) implies that $a^{-1}(0)=(a^{\dag})^{-1}(0)$. Therefore, according to \cite[Theorem 18(iii)]{B},
$a$ is EP.
\end{proof}

\indent In the following theorem  the condition of
being EP wil be considered in the context of $C^*$-algebras .\par

\begin{thm}\label{thm4.2} 
Let $A$ be a unital $C^*$-algebra and consider $a\in A$ such that $a^{ \dag}$ exists.
Then, the following statements are equivalent.\par
\noindent \rm (i) \it $a$ is EP,\par
\noindent \rm (ii) \it there exists $s\in A$ such that $s^{ -1}(0)=0$ and $a^*=sa$,\par
\noindent \rm (iii) \it there exist $s_1$ and $s_2 \in A$ such that $a^*=s_1a$ and $a=s_2a^*$,\par
\noindent \rm (iv) \it there exists $u\in A$ such that $uA=A$ and $a^*=au$,\par
\noindent \rm (v) \it there exist $u_1$ and $u_2\in A$ such that $a^*=au_1$ and $a=a^*u_2$,\par
\noindent \rm (vi) \it there exists $t\in A$ such that $t_{ -1}(0)=0$ and $a^*=at$,\par
\noindent \rm (vii) \it there exists $x\in A$ such that $Ax=A$ and $a^*=xa$,\par
\noindent \rm (viii) \it there exists $v\in A^{ -1}$ such that  $a^*a=vaa^*$,\par
\noindent \rm (ix) \it there exists $v_1\in A$ such that $v_1^{ -1}(0)=0$ and $a^*a=v_1aa^*$,\par
\noindent  \rm (x) \it there exist $v_2$ and $v_3\in A$ such that $a^*a=v_2 aa^*$
and $aa^* =v_3a^*a$,\par
\noindent \rm (xi) \it there exists $w\in A^{ -1}$ such that $a^*a=aa^*w$,\par
\noindent \rm (xii) \it there exists $w_1\in A$ such that $w_1A=A$ and  $a^*a=aa^*w_1$,\par
\noindent  \rm (xiii) \it there exist $w_2$ and $w_3\in A$ such that $a^*a= aa^*w_2$
and $aa^* =a^*aw_3$,\par
\noindent \rm (xiv) \it there exist $z_1$ and $z_2\in A$ such that $a^*a=az_1a^*$ and $aa^*=a^*z_2a$.\par
\noindent \rm (xv) \it  there exists $h_1\in A^{ -1}$ such that $a^*a= ah_1h_1^*a^*$,\par 
\noindent \rm (xvi) \it  there exists $h_2\in A$ such that $(h_2)^{ -1}(0)=0$ and $a^*a= ah_2h_2^*a^*$,\par 
\noindent \rm (xvii) \it  there exists $h_3\in A$ such that $h_3A=A$ and $a^*a= ah_3h_3^*a^*$.\par 
\end{thm}
\begin{proof}According to Theorem  \ref{thm4.1}(ii), $a$ is EP if and only if 
there exists $\tilde{ s}\in A$ such that $(\tilde{ s})^{ -1}(0)=0$ and $a^{\dag}=\tilde{ s}a$. Now well, since 
according to Lemma \ref{lem3.8}(ii) there exists $w\in A^{-1}$ such that $a^{\dag}=wa^*$, 
if $s=w^{-1}\tilde{ s}$, then  $s^{ -1}(0)=0$ and  $a^*=sa$.
On the other hand, if statement (ii) holds, then $(a^*)^{ -1}(0)=a^{ -1}(0)$. Thus,
according to  \cite[Theorem 3.1(iv)]{ K}, $a$ is EP.\par
\indent  To prove that statement (i) is equivalent to statements (iii)-(vii),
apply an argument similar to the one in the previous paragraph, using in particular the corresponding statements of  Theorem  \ref{thm4.1} and Lemma \ref{lem3.8}(i)-(ii).\par
\indent Next, if $a$ is EP, then, according to Lemma \ref{lem3.8}(i)-(ii), there exist $v,w\in A^{ -1}$ such that 
$a^{ \dag}=a^*v$ and $a^{\dag}=wa^*$. Then, according to  Lemma \ref{lem3.8},
$$
a^*a=w^{-1}a^{\dag}a= w^{-1}aa^{\dag}=w^{-1}aa^*v=w^{-1}vaa^*.
$$
Consequently, since $v$, $w\in A^{ -1}$, statement (viii) holds. In addition, it
is clear that statement (viii) implies statement (ix) and (x). On the other hand,
if one of these statements holds, then it is not difficult to prove that
$a^{ -1}(0)=(a^*)^{ -1}(0)$ (recall that $(aa^*)^{-1}(0)=(a^*)^{-1}(0)$ and $(a^*a)^{-1}(0)=a^{-1}(0)$). Therefore, 
according to \cite[Theorem 3.1(iv)]{ K}, $a$ is EP.\par

\indent To prove that statements (xi)-(xiii) are equivalent to the fact that
$a$ is EP, use an argument similar to the one in the previous paragraph and the identities $aa^*A=aA$ and $a^*aA=a^*A$ 
(\cite[Lemma 1.1(ii)]{DKS}).\par

\indent If $a$ is EP, according to  of \cite[Theorem 3.1(vi)-(vii)]{ K},
there exist  $m_1$, $m_2$, $m_3$ and $m_4$ such that 
$a=a^*m_1= m_2a^*$ and $a^*=am_3=m_4a$. As a result,
$a^*a= a(m_3m_2)a^*$ and $aa^*=a^*(m_1m_4)a$. Therefore, 
statement (xiv) holds. On the other hand, if this statement 
holds, then it is not difficult to prove that $a^{ -1}(0)=(a^*)^{ -1}(0)$. Consequentely,
according to \cite[Theorem 3.1(iv)]{ K}, $a$ is EP.\par

\indent Recall that if $a$ is EP, then according to \cite[Theorem 3.1(viii)]{ K},
 there exists $h_1\in A^{ -1}$ such that $a^*=ah_1$. Thus, statement (xv)
holds. Clearly, statement (xv) implies statements (xvi)-(xvii).\par

\indent On the other hand, if statement (xvi) holds, then
$a^*a= (ah_2)(ah_2)^*$. Consequently, a straightforward calculation proves that
$a^{ -1}(0)=((ah_2)^*)^{ -1}(0)$. However, since $(h_2)^{ -1}(0)=0$,
it is not difficult to prove that $a^{ -1}(0)=(a^*)^{ -1}(0)$. Thus,
according to \cite[Theorem 3.1(iv)]{ K}, $a$ is EP.\par

\indent Finally, if statement (xvii) holds, then $a^*a= (ah_3)(ah_3)^*$.
Therefore, since $h_3A=A$, $a^*A=ah_3A=aA$ (\cite[Lemma 1.1(ii)]{DKS}).
However, according to \cite[Theorem 3.1(vi)]{ K}, $a$ is EP.
\end{proof}

\section{\sfstp Factorization of the form $a=ucv$}
\
\indent In this section, given a unital Banach algebra $A$, EP elements of the form
 $a=ucv$ will be studied, $a,u,c,v\in A$; compare
 with  \cite[sections 1.2]{DKS}. However, in first place EP
Banach space operators will be characterized as block operators. 
Note that unlike to the Hilbert space context where the direct sum of two Hilbert spaces is again
a Hilbert space, there is no canonical way to give a norm to the direct sum of two Banach spaces. 
Moreover, given a Hilbert space $H$ and $H_1$ and $H_2$ two orthogonal and complementary
subspaces of $H$, $H_1\oplus H_2$ with its canonical Hilbert norm is isometrically isomorphic to
$H$. However, in the case of two closed and complementary subspaces $X_1$ and $X_2$ of a fixed Banach space $X$,
although the sum norn is equivalent to the original one, the natural identification between $X_1\oplus X_2$ and $X$ is not in general an isometry.
Consequently, since the norm is a key concept involved in
the notions of hermitian and Moore-Penrose invertible Banach space operators (\cite[Remark 4]{B}), and since an isometry is in general necessary to preserve 
the property of being hermitian and Moore-Penrose invertible (\cite[Remark 9]{B}), some results of \cite[section 3]{DKP}
can not be reformulated in terms of Banach space isomorphisms. Compare the results presented in this section with   \cite[section 3]{DKP}.\par

\indent In the following proposition, given two Banach spaces $X_1$ and $X_2$, $X_1\oplus_pX_2$ will denote the Banach space
$X_1\oplus X_2$ with the $p$-norm, $1\le p\le \infty$, see \cite[page 74]{C}. Note that the identity maps of
$X_1$ and $X_2$ will be denoted by $I_1$ and $I_2$ respectivelly.\par

\begin{pro} \label{prop5.1}Let $X_1$ and $X_2$ be two Banach space and consider $T_1\in L(X_1)$ a Banach
space isomorphism. Then, $T_1\oplus 0\in L(X_1\oplus_p X_2)$ has a Moore-Penrose inverse. Furthermore,
$(T_1\oplus 0)^{\dag}= (T_1^{-1}\oplus 0)$ and $T_1\oplus 0$ is an EP operator.
\end{pro}
\begin{proof} It is clear that  $T_1^{-1}\oplus 0$ is a normalized generalized inverse of $T_1\oplus 0$.
 Note that  $(T_1\oplus 0)(T_1^{-1}\oplus 0)=(T_1^{-1}\oplus 0)(T_1\oplus 0)=I_1\oplus 0=P_1$,
where $P_1\in L(X_1\oplus_p X_2)$ is the projection onto $X_1$.
Therefore, in order to conclude the proof, it is enough to prove that $P_1$ is an hermitian idempotent.
However, a straightforward calculation shows that $exp(itP_1)=P_2 + e^{it}P_1$, where $P_2\in L(X_1\oplus_p X_2)$
is the projection onto $X_2$ and $t\in \Bbb R$. As a result, $\parallel  exp(itP_1)\parallel =1$, for all $t\in \Bbb R$,
equivalently $P_1$ is an hermitian map.
\end{proof}

\indent Next the case of complementary closed subspaces of a given Banach spaces will be studied.\par

\begin{pro}\label{prop5.2}Let $X_1$ and $X_2$ be two Banach spaces and consider $T_1\in L(X_1)$
an iso-
morphic operator. Let $X$ be a Banach space and consider $T\in L(X)$ such that there exists
a linear and bounded isomorphism $J\colon X_1\oplus_1 X_2\to X$ with the property 
$T=J(T_1\oplus 0)J^{-1}$. Then, the following statements are equivalent.\par
\noindent \rm (i) \it $T$ has a Moore-Penrose inverse,\par
\noindent \rm (ii) \it $T$ is EP,\par
\noindent \rm (iii) \it $Q_1=J(I_1\oplus 0)J^{-1}\in L(X)$ is a hermitian idempotent,\par 
\noindent \rm (iv) \it $Q_2=J(0\oplus I_2)J^{-1}\in L(X)$ is a hermitian idempotent.\par 
\noindent In particular, if $J$ is an isometry, the four statements hold.\par
 \end{pro}
\begin{proof} It is not difficult to prove that $T'=J(T_1^{-1}\oplus 0)J^{-1}$ is a normalized generalized
inverse of $T$ and that $Q_1$ and $Q_2$ are the projections onto  the closed and complemented 
subspaces $J(X_1\oplus 0)$ and $J(0\oplus X_2)$ respectively. Furthermore, since $TT'=T'T=Q_1=I-Q_2$, 
statements (i)-(iv) are equivalent. Concerning the last statement, apply Proposition \ref{prop5.1} and \cite[Remark 9]{B}. 
\end{proof}

\indent As a result, the characterization of $EP$ bounded and linear maps as block operators
can be stated.\par

\begin{thm}\label{thm5.3} Let $X$ be a Banach space and consider $T\in L(X)$. Then, the following
statements are equivalent.\par
\noindent \rm (i) \it $T$ is EP,\par
\noindent\rm (ii) \it There exist two Banach spaces $X_1$ and $X_2$, $T_1\in L(X_1)$ an isomorphic
operator, and $J\colon X_1\oplus_1 X_2\to X$ a linear and bounded isomorphism such that
$T=J(T_1\oplus 0)J^{-1}$ and $J(I_1\oplus 0)J^{-1}\in L(X)$ is a hermitian idempotent.\par
\end{thm}
\begin{proof} According to Proposition \ref{prop5.2}, statement (ii) implies that $T$ is EP.
On the other hand, if $T$ is EP, according to \cite[Theorem 13]{B}, there exist $P\in L(X)$
a hermitian idempotent such that $N(P)=N(T)$ and $R(P)=R(T)$. Denote then
$X_1=R(P)$, $X_2=N(P)$ and $T_1=T\mid_{X_1}^{X_1}\colon X_1\to X_1$. It is clear that $T_1\in L(X_1)$
is an isomorphism. Moreover, $J\colon X_1\oplus_1 X_2\to X$ is the map
$J(x_1\oplus x_2)= x_1+x_2$. Since $J^{-1}\colon X\to X_1\oplus_1 X_2$ is such that
$J^{-1}=P\oplus (I-P)$, $I\in L(X)$ the identity map, $T=J(T_1\oplus 0)J^{-1}$ and $ J(I_1\oplus 0)J^{-1}=P$.
\end{proof}

\indent In the following theorem instead of isomorphic operators, injective and surjective bounded and linear
maps will be considered.\par

\begin{thm}\label{thm5.4}Let $X$ be a Banach space and consider $T\in L(X)$. Then, the
following statements are equivalent.\par
\noindent  \rm (i) \it  $T$ is EP,\par
\noindent \rm (ii) \it there exist Banach spaces $X_1$ and $X_2$, $T_1\in L(X_1)$ an isomorphism, $S\in L(X_1\oplus_1 X_2, X)$
injective, $U\in L(X,X_1\oplus_1 X_2)$ surjective, and $P\in L(X)$ a hermitian idempotent such that
$T=S(T_1\oplus 0)U$, $R(P)=S(X_1\oplus 0)$ and $N(P)=U^{-1}(0\oplus X_2)$.  
\end{thm}
\begin{proof} If $T$ is EP, then let $X_1$, $X_2$ and $T_1\in L(X_1)$ be as in Theorem \ref{thm5.3} and define
$S=J\in L(X_1\oplus_1 X_2, X)$ and $U=J^{-1}\in L(X,X_1\oplus_1 X_2)$, $J$ as in Theorem  \ref{thm5.3}.
Moreover, define $P=J(I_1\oplus 0)J^{-1}\in L(X)$. Since $R(P)=R(T)=S(X_1\oplus 0)$
and $N(P)=N(T)=U^{-1}(0\oplus X_2)$, statement (ii) holds. On the other hand, if statement (ii) holds,
then $P\in L(X)$ is a hermitian idempotent such that  $R(P)=R(T)$ and $N(P)=N(T)$.
Therefore, according to \cite[Theorem 13]{B}, $T$ is EP.
\end{proof}

\indent Next the conditions of Theorem \ref{thm5.3} will be weekend.\par

\begin{thm} \label{thm5.5} Let $X$ be a Banach space and consider $T\in L(X)$ such that
$T^{\dag}$ exists. Then, the following statements are equivalent.\par
\noindent \rm (i) \it $T$ is EP.\par

\noindent \rm (ii) (a) \it There exist Banach spaces $X_1$ and $X_2$, $A_1\in L(X_1)$
injective, $B_1\in L(X_2)$, $V_1$ and $W_1\in L(X_1\oplus_1 X_2, X)$, $V_1$ injective,
and $S_1\in L(X, X_1\oplus_1 X_2)$  such that $T=V_1(A_1\oplus 0)S_1$
and $T^{\dag}= W_1(B_1\oplus 0)S_1$. \par
\noindent\hskip.7truecm  \rm (b) \it There exist Banach spaces $X_3$ and $X_4$, $A_2\in L(X_3)$,
$B_2\in L(X_2)$ injective, $V_2$ and $W_2\in L(X_1\oplus_1 X_2, X)$, $W_2$ injective,
and $S_2\in L(X, X_1\oplus_1 X_2)$  such that $T=V_2(A_2\oplus 0)S_2$
and $T^{\dag}= W_2(B_2\oplus 0)S_2$. \par

\noindent \rm (iii) (a) \it There exist Banach spaces $Y_1$ and $Y_2$, $A_3\in L(Y_1)$
surjective, $B_3\in L(Y_2)$, $V_3\in L(Y_1\oplus_1 Y_2, X)$
and $S_3, S_4\in L(X, Y_1\oplus_1 Y_2)$, $S_3$ surjective, such that $T=V_3(A_2\oplus 0)S_3$
and $T^{\dag}= V_3(B_3\oplus 0)S_4$, \par
\noindent\hskip.7truecm  \rm (b) \it there exist Banach spaces $Y_3$ and $Y_4$, $A_4\in L(Y_3)$,
$B_4\in L(Y_4)$ surjective, $V_4\in L(Y_3\oplus_1 Y_4), X$
and $S_5$, $S_6\in L(X, Y_3\oplus_1 Y_4)$, $S_6$ surjective, such that $T=V_4(A_4\oplus 0)S_5$
and $T^{\dag}= V_4(B_4\oplus 0)S_6$. \par
\end{thm}
\begin{proof} According to Theorem \ref{thm5.3}, if $T$ is EP, then 
statement (ii) and (iii) hold. On the other hand, if statement (ii) hold,
then it no difficult to prove that $N(T)=N(T^{\dag})$. Therefore, according 
to \cite[Theorem 16(ii)]{B}, $T$ is EP. Similarly, if statement (iii) holds, then $R(T)=R(T^{\dag})$. Consequently,
according to  \cite[Theorem 16(iii)]{B}, $T$ is EP.
\end{proof}

\indent Next the Banach algebra frame will be considered.\par

\begin{thm}\label{thm5.6} Let $A$ be a unital Banach algebra, and consider $a\in A$ such that 
$a^{\dag}$ exists. Then, the following statement are equivalent.\par
\noindent \rm (i) \it The element $a$ is EP,\par
\noindent \rm (ii)\it there exist $b_1$, $c_1$, $d_1$, $f_1$ and $g_1\in A$ such that
$a=b_1c_1g_1$, $a^{\dag}= f_1d_1g_1$, $(c_1)^{-1}(0)=(d_1)^{-1}(0)$, and $(b_1)^{-1}(0)=(f_1)^{-1}(0)=0$,\par 
\noindent \rm (iii) \it there exist $h_1$, $k_1$, $l_1$, $m_1$ and $n_1\in A$ such that
$a=h_1k_1l_1$, $a^{\dag}= h_1m_1n_1$, and $l_1A=A=n_1A$, and $k_1A=m_1A$.\par

\noindent \rm (iv)\it there exist $b_2$, $c_2$, $d_2$, $g_2$ and $g_3\in A$ such that
$a=b_2c_2g_2$, $a^{\dag}= b_2d_2g_3$, $(c_2)_{-1}(0)=(d_1)_{-1}(0)$, and $(g_2)_{-1}(0)=(g_3)_{-1}(0)=0$,\par 
\noindent \rm (v) \it there exist $h_2$, $h_3$, $k_2$, $l_2$ and $m_2\in A$ such that
$a=h_2k_2l_2$, $a^{\dag}= h_3m_2l_2$, and $Ak_2=Am_2$, and $Ah_2=Ah_3=A$.\par

\end{thm}
\begin{proof} If $a\in A$ is EP, then consider $b_1=g_1=f_1=e$,
$c_1=a$ and $d_1=a^{\dag}$ and use \cite[Theorem 18(iii)]{B}. 
On the other hand, if  statement (ii) holds, a straightforward
calculation shows that $a^{-1}(0)=(a^{\dag})^{-1}(0)$
which, according again to \cite[Theorem 18(iii)]{B}, implies that $a$ is EP.\par

\indent A similar argument, using in particular that
necessary and sufficient for $a\in A$ to be EP is the fact that 
$aA=a^{\dag}A$ (\cite[Theorem 18(iv)]{B}), proves that statements (i) and (iii)
are equivalent.\par

\indent To prove that statement (i) and statements (iv)-(v) are equivalent,
apply arguments similars to the ones used to prove that the condition
of being EP is equivalent to statements (ii)-(iii), using in particular
 \cite[Theorem 18(viii)-(ix)]{B} instead of statements
\cite[Theorem 18(iii)-(iv)]{B}.
\end{proof}

\indent Note that under  the hypothesis of Theorem \ref{thm5.6},
others statements equivalent to the condition of being EP can be
obtained if instead of $a$ and $a^{\dag}$, $a^{\dag}a$ and 
$aa^{\dag}$ are considered. In fact, using the fact that
$(a^{\dag}a)^{-1}(0)=a^{-1}(0)$ and  $(aa^{\dag})^{-1}(0)=(a^{\dag})^{-1}(0)$,
arguments similar to the ones in  Theorem \ref{thm5.6} prove the corresponding statements.

\vskip.5truecm

\noindent Enrico Boasso\par
\noindent E-mail address: enrico\_odisseo@yahoo.it
\end{document}